\documentclass[letterpaper, 10 pt, conference]{ieeeconf}  
\IEEEoverridecommandlockouts                              
\let\labelindent\relax

\usepackage[utf8]{inputenc}
\usepackage{graphics} 
\usepackage{amsmath} 
\usepackage{amssymb}  
\usepackage{amsfonts}
\usepackage{comment}
\usepackage{enumitem}
\usepackage{cases,xspace,enumitem}

\usepackage{version}
\includeversion{arxiv}
\excludeversion{lcss}

\usepackage{tikz}
\usetikzlibrary{positioning}

\usepackage[prependcaption,colorinlistoftodos]{todonotes}

\definecolor{gnred}{RGB}{255,91,89}

\definecolor{gnred1}{RGB}{71,0,0} 
\definecolor{gnred2}{RGB}{117,0,0} 
\definecolor{gnred3}{RGB}{164,0,0} 
\definecolor{gnred4}{RGB}{211,0,0} 
\definecolor{gnred5}{RGB}{255,0,0} 
\definecolor{gnred6}{RGB}{255,42,34} 
\definecolor{gnred7}{RGB}{255,91,89} 

\definecolor{gnblue1}{RGB}{0,36,71}   
\definecolor{gnblue2}{RGB}{0,60,118}  
\definecolor{gnblue3}{RGB}{0,85,164}
\definecolor{gnblue4}{RGB}{0,108,212}
\definecolor{gnblue4}{RGB}{0,108,212}
\definecolor{gnblue5}{RGB}{0,133,255}  
\definecolor{gnblue6}{RGB}{35,156,255} 
\definecolor{gnblue7}{RGB}{88,177,255} 

\definecolor{gnbrown1}{RGB}{71,27,0}  
\definecolor{gnbrown2}{RGB}{117,45,0} 
\definecolor{gnbrown3}{RGB}{164,62,0} 
\definecolor{gnbrown4}{RGB}{211,80,0} 
\definecolor{gnbrown5}{RGB}{255,97,0} 
\definecolor{gnbrown6}{RGB}{255,127,26} 
\definecolor{gnbrown7}{RGB}{255,155,86} 

\renewcommand{\theenumi}{(\roman{enumi})}

\newcommand\Item[1][]{%
  \ifx\relax#1\relax  \item \else \item[#1] \fi
  \abovedisplayskip=0pt\abovedisplayshortskip=0pt~\vspace*{-\baselineskip}}

\usepackage{hyperref}
\hypersetup{
 colorlinks=true,
 linkcolor=blue,
 filecolor=magenta,      
 urlcolor=cyan,
 pdftitle={Overleaf Example},
 pdfpagemode=FullScreen,
 }

\newcommand{\e}{\operatorname{e}}

\usepackage{amsthm}

\newtheoremstyle{ieeeconf}
 {0pt}   
 {0pt}   
 {\normalfont}  
 {\parindent}       
 {\itshape} 
 {:}         
 { } 
 {\thmname{#1} \thmnumber{#2}\thmnote{ (#3)}} 
\makeatletter
\renewenvironment{proof}[1][\proofname]{\par
  \pushQED{\qed}%
  \normalfont \topsep\z@
  \trivlist
  \item[\hskip2em
        \itshape
 #1\@addpunct{:}]\ignorespaces
}{%
  \popQED\endtrivlist\@endpefalse
}
\makeatletter

\theoremstyle{ieeeconf}

\newtheorem{defn}{Definition}
\newtheorem{thm}{Theorem}
\newtheorem{lemma}[thm]{Lemma}

\newtheorem{cor}[thm]{Corollary}
\newtheorem{remark}[thm]{Remark}

\newcommand{\R}{\mathbb{R}}
\newcommand{\real}{\mathbb{R}}

\newcommand{\norm}[2]{\|#1\|_{#2}}

\newcommand{\tr}{\operatorname{tr}}

\newcommand{\setdef}[2]{\{#1 \, | \, #2\}}

\newcommand{\map}[3]{#1 \colon #2 \rightarrow #3}

\newcommand{\argmin}[1]{\underset{#1}{\textup{arg\,min}}}

\newcommand{\realextended}{\overline{\real}}

\newcommand{\mcD}{\mathcal{D}}

\newcommand{\mcX}{\mathcal{X}}

\newcommand{\xstar}{x^{\star}}

\newcommand{\softt}[2]{\operatorname{soft}_{#1}({#2})}
\newcommand{\softtbig}[2]{\operatorname{soft}_{#1}\bigl({#2}\bigr)}
\newcommand{\Softt}[2]{\operatorname{S}_{#1}({#2})}
\newcommand{\Softtbig}[2]{\operatorname{S}_{#1}\bigl({#2}\bigr)}

\newcommand{\sign}[1]{\operatorname{sign}({#1})}

\newcommand{\proj}[1]{\mathrm{proj}_{#1}}
\newcommand{\prox}[1]{\mathrm{prox}_{#1}}
\newcommand{\Fprox}{\operatorname{F_{\prox{}}}}
\newcommand{\Fgd}{\operatorname{F_{gd}}}

\DeclareMathOperator{\dom}{Dom}

\newcommand{\change}[1]{\textcolor{black}{#1}}
\newcommand{\sashachange}[1]{\textcolor{black}{#1}}

\title{Proximal Gradient Dynamics: \\ Monotonicity, Exponential Convergence, and Applications}

\author{Anand Gokhale, Alexander Davydov, Francesco Bullo \thanks{This work
    was in part supported by AFOSR project FA9550-21-1-0203 and NSF
    Graduate Research Fellowship under Grant 2139319.  The authors are with
    the Center for Control, Dynamical Systems, and Computation, UC Santa
    Barbara, Santa Barbara, CA 93106 USA.  {\tt\small
      \{anand\_gokhale,davydov,bullo\}@ucsb.edu}.}  }

\begin{document}

\maketitle

\thispagestyle{empty}
\pagestyle{empty}

\begin{abstract}
  In this letter we study the proximal gradient dynamics. This
  recently-proposed continuous-time dynamics solves optimization problems
  whose cost functions are separable into a nonsmooth convex and a smooth
  component.
  First, we show that the cost function decreases monotonically along the
  trajectories of the proximal gradient dynamics.
  We then introduce a new condition that guarantees exponential convergence
  of the cost function to its optimal value, and show that this condition
  implies the proximal Polyak-\L{ojasiewicz} condition. 
  We also show that the proximal Polyak-\L{ojasiewicz} condition guarantees exponential convergence of the cost function.
  Moreover, we extend these results to time-varying optimization problems,
  providing bounds for equilibrium tracking.
  Finally, we discuss applications of these findings, including the LASSO
  problem, certain matrix based problems \change{and a numerical experiment on a feed-forward neural network.}
\end{abstract}

\section{Introduction}
Optimization problems are central to various domains of engineering and
data science. Many of these problems can be framed as minimizing a smooth
cost function combined with a nonsmooth but convex regularizer that promotes
structure in the solution. For example, the addition of the $\ell_1$ norm as a
regularizer induces sparsity, a key principle in compressed sensing and
sparse coding~\cite{AB-MT:08}. In other cases, convex constraints can be
incorporated using indicator functions. A common approach for solving such
problems, where the cost function is separable into a smooth, possibly
nonconvex, and a nonsmooth convex component, is the proximal gradient
algorithm~\cite{NP-SB:14}. This (discrete-time) algorithm gained prominence
in sparse encoding applications~\cite{AB-MT:08}. A list of widely used
regularizers and their associated proximal operators is available
in~\cite[Chapter 10]{AB:17}.

Analyzing the continuous-time versions of discrete-time optimization
algorithms has garnered much interest since the seminal work
in~\cite{KJA-LH-HU:58}.  Continuous-time dynamics have proven useful across
various engineering disciplines, including notable contributions from
Hopfield and Tank in dynamical neuroscience~\cite{JJH-DWT:85}, Kennedy and
Chua in analog circuit design~\cite{MPK-LOC:88}, and Brockett in systems
and control theory~\cite{RWB:91}.
\change{There are several compelling reasons to investigate continuous-time
version of the proximal gradient algorithm. 
First, these dynamics establish a connection between machine learning and 
dynamical neuroscience. The inclusion of a nonsmooth convex component allows 
for the incorporation of regularizers into neuroscientific loss landscapes, 
similar to those used in machine learning. Second, proximal algorithms form the cornerstone of modern optimization algorithms.
In Model Predictive Control (MPC), solution methods like Sequential Quadratic Programming (SQP) and Alternating Direction Method of Multipliers (ADMM) rely on proximal gradient updates.
Insights from continuous-time algorithms also inform discrete-time approaches.}


\textit{Related Work: } The analysis of algorithms through the lens of systems theory is of much recent interest~\cite{FD-HZ-GB-SB-JL-MM:24}. Proximal gradient algorithms have been extensively
studied in the discrete-time optimization setting~\cite{NP-SB:14}.  A key
result from this analysis is the identification of the
Polyak-\L{ojasiewicz} (PL) condition, initially introduced for gradient
descent in~\cite{BTP:63} and later extended to proximal methods
in~\cite{HK-JN-MS:16}.  When the cost function satisfies this condition in
discrete time, proximal gradient algorithms achive linear convergence. The
authors of~\cite{HK-JN-MS:16} also discuss several existing conditions in
literature and their relationships with the PL condition. Notably, they
propose proximal versions of the PL and Kurdyka \L{ojasiewicz} (KL)
conditions and show their equivalence.

The study of optimization problems as continuous-time dynamical systems has
a long history~\cite{KJA-LH-HU:58}. 
  \change{More recently, they have been used to solve primal dual dynamics~\cite{GQ-NL:19}, nonlinear programming problems~\cite{AA-JC:23} and saddle point problems~\cite{AC-BG-JC:17}}.
  The continuous-time proximal gradient dynamics were
first proposed in~\cite{SHM-MRJ:21} where it was shown that the PL
condition implies the exponential convergence of the forward-backward
envelope, even in the absence of strong convexity.

\textit{Contributions:} In this work, we analyze the convergence properties
and applications of proximal gradient dynamics for optimization problems
with cost functions that combine a nonsmooth convex term and a potentially
nonconvex smooth term \change{in continuous time. Our analysis is based on studying the convergence of the cost function, as in many applications, the cost is typically measurable and interpretable.} Our first contribution in Theorem~\ref{thm:decreasing-energy-smooth} demonstrates that the cost
function decreases monotonically along the trajectories of the proximal
gradient dynamics, even in a general nonconvex setting. \change{This finding draws a parallel with the behavior of gradient descent dynamics in continuous time. We also some general properties of subgradients of convex functions, which facilitate the proof of Theorem~\ref{thm:decreasing-energy-smooth}.}
We then identify a natural condition that ensures exponential convergence in continuous time; we show that our condition implies the proximal KL condition. Next, we examine the PL condition for proximal gradient dynamics in Theorem~\ref{thm:PL_cts_time} and
prove that it guarantees exponential convergence of the cost function, \change{which is an exact analog to the result in discrete time. Prior work~\cite{SHM-MRJ:21} showed that, under the PL condition, the forward backward envelope of the cost function exponentially decays, but this function is not necessarily interpretable in all cases.} Additionally, we derive Lipschitz-based bounds for tracking the
optimal function value in parameter-varying problems under the PL
condition in Theorem~\ref{thm:PL_time_varying}. Finally, we explore applications of these results in sparse
reconstruction, neural networks, and the analysis of some matrix based problems.

The letter is structured as follows:
Section~\ref{sec:prob_and_prelims} covers preliminaries and formulates the
optimization problem. Section~\ref{sec:global_convergence} presents our
main result on the convergence of the proximal dynamics. We discuss exponential stability under the proximal PL
condition in Section~\ref{sec:PL_condn}, followed by
applications of our results in Section~\ref{sec:apps}.

\section{{Problem Formulation and Preliminaries}}\label{sec:prob_and_prelims}
\emph{Notation:} We define $\realextended := [-\infty,\infty]$. We denote a norm on $\R^n$ as $\norm{\cdot}{}$. For $p \in [1,\infty]$, we denote the $p-$norm on $\R^n$ as $\norm{\cdot}{p}$. The inner product is denoted by $\langle\cdot,\cdot\rangle$. We deonte the matrix Frobenius norm by $\norm{\cdot}{F}$, and the trace norm as $\norm{X}{*}$. We denote the signum function as $\sign{\cdot}$. We denote the gradient of a differentiable function $f$ at $x$ by $\nabla f(x)$, and the set of subdifferentials of a function $g$ at $x$ by $\partial g(x)$. For $g:\R\to\realextended$, we denote $\dom g := \setdef{x \in \real^n}{g(x) \neq \pm \infty}$. We denote the projection operation onto $C$ by $\proj{C}(\cdot)$. 

Next, we recall some definitions.
\begin{defn}[Dini Derivative]
 For $\map{\varphi}{(a,b)}{\real}$, we define the upper right Dini derivative at $t \in (a,b)$ as
\begin{align}
 D^+ \varphi(t) := \limsup_{h \to 0^+} \frac{\varphi(t+h) - \varphi(t)}{h}.
\end{align}
\end{defn}
\begin{defn}[L-smoothness]
 A differentiable function ${f:\R^n \to \R}$ is said to be $L-$smooth in a norm $\norm{\cdot}{}$, if
\begin{align}
 \norm{\nabla f(x) - \nabla f(x)}{} \leq L\norm{x-y}{}.
\end{align}
\end{defn}

In this work, we consider the optimization problem,
\begin{align}\label{eq:opt_problem}
 \min_{x} f(x) + g(x) = F(x),
\end{align}
where $f:\R^n \rightarrow \R$ is continuously differentiable and $g: \R^n \rightarrow \realextended$ is a convex closed proper (CCP) function which is potentially nondifferentiable. We begin by introducing the proximal operator and the proximal gradient dynamics.

\subsection{Proximal operator and the proximal gradient dynamics}
The proximal operator of a closed, convex, and proper function $\map{g}{\real^n}{\realextended}$ is defined as
\begin{align}
 \prox{\alpha g}(x) = \argmin{u \in \real^n}\bigl(\alpha g(u) + \frac{1}{2}\norm{u - x}{2}^2\bigr).
\end{align}

The proximal gradient dynamics proposed in~\cite{SHM-MRJ:21}, are given by
\begin{align}\label{eq:prox_dynamics}
 \dot x = -x + \prox{\alpha g}\bigl(x - \alpha \nabla f(x) \bigr) = F_{\prox{}}(x).
\end{align}
The authors of \cite{SHM-MRJ:21} prove that under assumptions on the strong convexity and smoothness of $f$, these dynamics are globally exponentially converge to the solution of optimization problem~\eqref{eq:opt_problem}, as a consequence of the nonexpansive nature of the proximal operator. 

\subsection{Polyak-\L{ojaciewicz} and related conditions}
The PL condition~\cite{BTP:63} was initially proposed to show linear convergence in discrete time setups for gradient descent.
\begin{defn}[PL condition]\label{def:PL_condition}
 A differentiable function $\map{f}{\real^n}{\real}$ which has a minimum at $\xstar \in \R^n$ is said to satisfy the PL condition if there exists a $\mu > 0$ such that
  \begin{align}\label{eq:PL_condition}
 \frac{1}{2} \norm{\nabla f(x)}{}^2 \geq \mu \bigl(f(x) - f(\xstar) \bigr).
  \end{align}
\end{defn}
The PL condition is weaker than the usual assumption of convexity. A discussion of the different conditions for convergence of gradient descent in discrete time is discussed in~\cite{HK-JN-MS:16}. Further~\cite{HK-JN-MS:16} also proposes an extension of the PL condition for the proximal case, under the assumption of $L-$smoothness, and showed its equivalence to the KL condition. 
\begin{defn}[Proximal PL Condition~\cite{HK-JN-MS:16}] $\map{f}{\real^n}{\real}$ and $\map{g}{\real^n}{\realextended}$ are said to satisfy the proximal PL condition if there exists a ${\mu > 0}$ and ${\alpha > 0}$, for all $x$, such that
\begin{align}\label{eq:Prox_PL_condition}
 \frac{1}{2} \mcD_g(x,\alpha) \geq \mu(F(x) - F(\xstar)),
\end{align}
where $F(\xstar)$ is the minimum value of $F$, and
\begin{align*}
   &\mcD_g(x,\alpha) =  \\
   & \frac{2}{\alpha} \max_{y\in \real^n} \left\{\langle\nabla f(x), x-y \rangle - \frac{1}{2\alpha}\norm{x -y}{2}^2 + g(x) - g(y)\right\}.
\end{align*}
\end{defn}
\begin{remark}
 Note that the original definition in~\cite{HK-JN-MS:16} enforces $\alpha =1/L$, where $L$ is the smoothness parameter of $f$. We relax this constraint in order to achieve a more general bound.  
\end{remark}

\begin{defn}[Proximal KL condition~\cite{HK-JN-MS:16}]
The KL condition with exponent $\frac{1}{2}$ holds if there exists $\hat \mu > 0$ such that for all $x$,
\begin{align}\label{eq:KL_condition}
 \min_{s \in \partial F(x)} \norm{s}{2}^2\geq 2\hat \mu(F(x) - F(\xstar)),
\end{align}
where $F(x) = f(x) + g(x)$ and the $\partial F(x)$ denotes the set $\setdef{\nabla f(x) + v}{v \in \partial g(x)}$.
\end{defn}

\section{Global convergence of the proximal gradient dynamics}\label{sec:global_convergence}

We begin with our first main result which shows that under the proximal gradient dynamics, the cost function decreases along trajectories of the system. First, we consider the case where $g$ is finite-valued. 

\begin{thm}[Nonincreasing cost function under proximal gradient dynamics] \label{thm:decreasing-energy-smooth}
 For the optimization problem~\eqref{eq:opt_problem}, let the following assumptions hold true.
    \begin{enumerate}[label=\textup{(A\arabic*)},leftmargin=*]
      \item $f$ is continuously differentiable.
      \item $g$ is convex, closed, and proper (CCP), and takes finite values for all finite $x$.
      \item A global minimizer $\xstar$ exists, and $f(\xstar) + g(\xstar) = F^*$.
  \end{enumerate}
 Then, for the proximal gradient dynamics~\eqref{eq:prox_dynamics}:
  \begin{enumerate}
    \item \label{fact:prox-descent-i} the function $V(x) = f(x) + g(x) -
      F^*$ is nonincreasing along solutions and $D^+ V(x) \leq
      -\frac{1}{\alpha} \norm{\Fprox(x)}{2}^2$;
      
    \item \label{fact:prox-descent-ii} each trajectory converges to the set
      of stationary points
      \begin{align*}
        \Omega = \setdef{x}{0 \in \nabla f(x) + \partial g(x)};
      \end{align*}
    \item if additionally, there exists $\mu > 0$ such that for every $x$,
      \begin{align}\label{eq:better_than_PL}
        \frac{1}{2}\norm{\Fprox(x)}{2}^2 \geq \mu\alpha^2 (F(x) - F^*),
      \end{align}
      then each trajectory $x(\cdot)$ satisfies
      \begin{align*}
        F(x(t)) - F^* \leq \e^{-2\mu\alpha t} (F(x(0)) - F^*),
      \end{align*}
      in other words, $F(x(\cdot))$ globally exponentially converges to $F^*$
      with rate $2\mu\alpha > 0$.
    \end{enumerate}
\end{thm}
\begin{proof}
  For the proximal gradient dynamics~\eqref{eq:prox_dynamics}, let
  $\map{\phi_h}{\real^n}{\real^n}$ denote the $h$-time flow map of the
  dynamics.  Under the assumption of continuously differentiable $f$,
  $\Fprox$ is locally Lipschitz, and thus the flow map $\phi_h$ exists for
  sufficiently small $h \in \real$ forward in time. We
  consider the function $V(x) = F(x) - F^*$, and consider its Dini
  derivative with respect to time, denoting $x(t)$ by $x$ for brevity,
  setting ${z = \prox{\alpha g}(x - \alpha \nabla f(x))}$. \sashachange{We will let $\mathcal{A}$ be the set of nonincreasing sequences with positive entries that converge to $0$ and let $\mcX$ be the set of (vector-valued) sequences converging to $x$. For a sequence $\{x_k\}_{k=1}^\infty \in \mcX$, define $\{s_{x_k}\}_{k=1}^\infty$ to be the sequence defined by $s_{x_k} = \argmin{s \in \real^n} \|s\|_2^2$ s.t. $s \in \partial g(x_k)$ for all $k$. Then,}
\begin{align}
	&\sashachange{D^+ V(x)} \sashachange{= \limsup_{h \to 0^+} \frac{V(\phi_h(x)) - V(x)}{h}} \nonumber\\
	&\sashachange{\overset{(!)}{\leq} \nabla f(x)^\top (z-x) + \limsup_{h \to 0^+} \frac{g(\phi_h(x)) - g(x)}{h}} \nonumber\\
	&\sashachange{\overset{(\star)}{\leq} \nabla f(x)^\top (z-x) + \limsup_{h \to 0^+} \frac{s_{\phi_h(x)}^\top (\phi_h(x)-x)}{h}} \nonumber\\
	&\sashachange{\overset{(\ae)}{\leq} \nabla f(x)^\top (z-x) + \limsup_{h \to 0^+} s_{\phi_h(x)}^\top (z-x)} \nonumber\\
	&\sashachange{\overset{(\diamond)}{=} \nabla f(x)^\top (z-x) + \sup_{\{h_k\} \in \mathcal{A}}\limsup_{k \to \infty} s_{\phi_{h_k}(x)}^\top(z-x)} \nonumber\\
	&\sashachange{\overset{(\pi)}{\leq} \nabla f(x)^\top (z-x) + \sup_{\{x_k\} \in \mathcal{X}}\limsup_{k \to \infty} s_{x_k}^\top(z-x)} \nonumber\\
	&\sashachange{\overset{\varpi}{\leq} \nabla f(x)^\top (z-x) + \max_{{s_x \in \partial g(x)}}s_{x}^\top (z-x),}\label{eq:convergence_argument_start}
\end{align}
\sashachange{where $s_{\phi_{h}(x)}$ is the minimum-norm element in $\partial g(\phi_h(x))$. We clarify each of the above (in)equalities. Inequality $(!)$ holds because $D^+(\varphi + \theta)(t) \leq D^+\varphi(t) + D^+\theta(t)$ and $D^+ f(x) = \nabla f(x)^\top (z-x)$. Inequality $(\star)$ holds because $g$ is convex and $g(x) \geq g(\phi_h(x)) + s_{\phi_h(x)}^\top (x - \phi_h(x))$ for all $s_{\phi_h(x)} \in \partial g(\phi_h(x))$. Inequality $(\ae)$ holds because $x - \phi_{h}(x) = h(z - x) + \mathcal{O}(h^2)$. Finally, inequality $(\varpi)$ is technical, but is essentially a semicontinuity argument and is proved in Lemma~\ref{lem:ccp_subdifferential_continuity}}
 for $s_x \in \partial g(x)$. By convexity of $g(x)$, 
  \begin{align}
 D^+ V(x(t))  \leq \nabla f(x)^\top (z - x) + g(z) - g(x).\label{eq:useful_for_PL}
  \end{align}
 From convexity of $g$, we also have, $g(z) \leq g(x) -  s^\top (x-z)$,
 for all $s \in \partial g(z)$. Using this, we obtain,
  \begin{align*}
 D^+ V(x(t))  &\leq \nabla f(x)^\top (z - x) + g(x) - s^\top(x - z) - g(x) \\
      &=\nabla f(x)^\top (z - x) - s^\top(x - z).
  \end{align*}
 Upon an application of Lemma~\ref{lem:prox-optimality} with $v = x - \alpha\nabla f(x)$, we may choose ${s = - \alpha^{-1}(z - (x - \alpha\nabla f(x)))}$. Therefore,
  \begin{align}
      &D^+ V(x(t)) \nonumber \\
      &\leq \nabla f(x)^\top (z-x) + \alpha^{-1}(z - (x - \alpha \nabla f(x)))^\top (x-z) \nonumber \\
      &= - \frac{1}{\alpha} \norm{x-z}{2}^2 = - \frac{1}{\alpha} \norm{\Fprox(x)}{2}^2. \label{eq:bound_on_D+F}
  \end{align}
 This proves statement~\ref{fact:prox-descent-i}.  Applying a generalization of LaSalle's invariance principle, extended to Dini derivatives, the system converges to the set where $ D^+ V(x) = 0$, which may be characterized by the set $\setdef{x}{x = z}$. Upon an application of Lemma~\ref{lem:prox-optimality} with $v = x - \alpha\nabla f(x)$, we can see that $\setdef{x}{x = z} = \Omega$.

 The final statement follows from substituting~\eqref{eq:better_than_PL} in~\eqref{eq:bound_on_D+F}, and using Gr\"{o}nwall's lemma.
  \end{proof}
\begin{remark}
 Our choice of scaling $\mu$ in Equation~\eqref{eq:better_than_PL} is motivated by an observation that is a natural extension to the existing PL condition~\eqref{eq:PL_condition} for gradient descent. Let us consider the gradient flow given by $\dot x = - \alpha\nabla f(x) = \Fgd(x)$. The PL condition for gradient descent presented in Definition~\ref{def:PL_condition} can be rewritten as, 
  \begin{align*}
 \frac{1}{2}\norm{\Fgd(x)}{}^2 \geq \mu\alpha^2 (f(x) - f(\xstar)).
  \end{align*}
 It is easy to see that under the gradient flow, the PL condition forces an exponential convergence rate of $2\mu\alpha$. We note that both the condition in terms of the flow and the rate are exactly equivalent to what we obtain in the proximal case.
\end{remark}

\begin{remark}
 Unlike their discrete time counterparts, the proximal gradient dynamics are guaranteed to converge for all $\alpha > 0$. In the discrete case, $\alpha$ is forced to be in $(0,1/L]$. We also do not assume that $f$ is $L-$smooth in our proof.
\end{remark}

\begin{arxiv}
  \begin{remark}
    In prior work~\cite{JB:03}, a similar result was proved for the special case where $g(x)$ is the indicator of a convex set. Our result generalizes to all proximal operators.
  \end{remark}
\end{arxiv}

\begin{lemma}
 Under the same assumptions as Theorem~\ref{thm:decreasing-energy-smooth}, the condition~\eqref{eq:better_than_PL} implies
 the KL condition~\eqref{eq:KL_condition}, and by extension also implies the PL condition~\eqref{eq:Prox_PL_condition}.
\end{lemma}
\begin{lcss}
  \begin{proof}
    \change{The proof follows from a Cauchy-Schwarz argument, and is available in the arxiv version\footnote{\href{https://arxiv.org/pdf/2409.10664}{https://arxiv.org/pdf/2409.10664}}.}
  \end{proof}
  \end{lcss}
\begin{arxiv}
\begin{proof}
 For every $x$, consider, ${s^\top(-x + z)}$, where ${z = \prox{\alpha g}(x - \alpha \nabla f(x))}$, and $s_x\in \partial F(x)$. Following the arguments in the proof of Theorem~\ref{thm:decreasing-energy-smooth} from~\eqref{eq:convergence_argument_start} to~\eqref{eq:bound_on_D+F}, we have
  \begin{align*}
    &s_x^\top(-x + z) \leq -\frac{1}{\alpha} \norm{x - z}{2}^2\\
 &\quad\implies \norm{s_x}{2}\norm{x - z}{2} \geq \frac{1}{\alpha} \norm{x - z}{2}^2,
  \end{align*}
 where the inequality follows through Cauchy Schwarz. Note that $x = z$ only occurs at the optimal values of $x$, i.e, where $F(x) = F^*$. In the case where $F(x) = F^*$, $\partial F(x) = 0$ through an optimality argument. Therefore the KL condition with exponent half is true. In the case where $x\neq z$, we have for every valid $s_x\in \partial F(x)$,   
  \begin{align*}
 \norm{s_x}{2} \geq \min_{s \in \partial F(x)} \norm{s}{2} \geq \frac{1}{\alpha} \norm{x - z}{2}.
  \end{align*}
 The KL condition follows upon squaring both sides and substituting~\eqref{eq:better_than_PL}.
\end{proof}
\end{arxiv}
In Theorem~\ref{thm:decreasing-energy-smooth}, we show that costs along the trajectories of the proximal gradient dynamics are decreasing, in cases where $g$ is finite-valued. In many practical examples, for instance, in the case where $g$ is an indicator function, this assumption is not met. However, in Theorem~\ref{thm:decreasing-energy-nonsmooth}, we will show that the cost remains non-increasing even for non-finite-valued $g$. 

For the proximal gradient dynamics~\eqref{eq:prox_dynamics}, recall that $\map{\phi_h}{\real^n}{\real^n}$ denotes the $h$-time flow map of the dynamics. We consider the evolution of $F(x) = f(x) + g(x)$ along trajectories of the system by studying the evolution of $t \mapsto F(\phi_t(x))$ for all $x$. 

Now, we are ready to discuss the convergence of $F(\phi_t(x))$ in the case where $g$ takes infinite values. 

\begin{thm}[Nonincreasing cost for proximal gradient dynamics with infinite g]\label{thm:decreasing-energy-nonsmooth}
   Let $\map{f}{\real^n}{\real}$ be continuously differentiable and $\map{g}{\real^n}{\realextended}$ be CCP. Let $x \in \real^n$ and let $\phi_t(x)$ be the $t$-time flow map of the proximal gradient dynamics~\eqref{eq:prox_dynamics}. Then
   the map $t \mapsto F(\phi_t(x))$ is nonincreasing. Further, it is either always equal to $+\infty$, is always finite, or is infinite for a finite amount of time and then becomes finite. 
  \end{thm}
  \begin{proof}

   For a lower semicontinuous function, we can establish that it is nonincreasing provided that its upper right Dini derivative is nonpositive~\cite[Theorem 1.13]{GG-SK:92}. We note that $F(\phi_t(x))$ is lower semicontinuous, as $g$ is CCP and hence lower semicontinuous, which implies $F$ is as well, which in turn implies that the composition $t \mapsto F(\phi_t(x))$ is lower semicontinuous. Now, we analyse each case separately. If $F(\phi_t(x)) = +\infty$ for all $t \geq 0$, there is nothing to prove. Suppose alternatively that it is always finite. Then we know that the map $F(\phi_t(x))$ is lower semicontinuous and we can compute its Dini derivative, and show that it is nonpositive. Upon noting that the constant $F^*$ does not affect the Dini derivative $D^+ F(\phi_t(x))$, we may use the same arguments obtained in Theorem~\ref{thm:decreasing-energy-smooth}, to obtain 
    \begin{align*}
   D^+ F(\phi_t(x)) &\leq  -\frac{1}{\alpha}\|\phi_t(x) - z\|_2^2 \leq 0.    
    \end{align*}
   In other words, when $F(\phi_t(x))$ is finite, it is nonincreasing. In the case that it is infinite and then becomes finite after some time horizon, the same argument as above applies on the interval of time when $F(\phi_t(x))$ is finite. Due to the forward-invariance of $\dom g$ \change{and Lemma~\ref{lemma:fwd-invariance}}, $F(\phi_t(x))$ is either always equal to $+\infty$, is always finite, or is infinite for a finite amount of time and then becomes finite. In other words, once $F(\phi_t(x))$ is finite, it can never become $+\infty$.
  \end{proof}
\begin{remark}
 Due to the arguments made in this theorem, in the special case where $g$ is the indicator of a convex set, if the initial state is within the feasible set, then trajectories of proximal gradient dynamics always remain within the feasible set. On the other hand, if $g$ is the indicator function of the origin, then the proximal gradient dynamics are ${\dot{x} = -x}$ and for any $x \neq 0$, $F(\phi_t(x))$ is infinite for all $t \geq 0$ even though $\phi_t(x) \to 0$ as $t \to \infty$. 
\end{remark}

\section{Exponential convergence of the cost}\label{sec:PL_condn}
We discuss sufficient conditions for exponential convergence of the cost in~\eqref{eq:opt_problem} under our dynamics~\eqref{eq:prox_dynamics}.
\begin{thm}[Exponential convergence under the PL condition]\label{thm:PL_cts_time}
 For the optimization problem~\eqref{eq:opt_problem} under the proximal gradient dynamics~\eqref{eq:prox_dynamics}, such that $f(x) + g(x)$ satisfies the proximal PL condition~\eqref{eq:Prox_PL_condition} with constants $\mu > 0,\alpha > 0$ and $g$ is CCP,  we have guaranteed global exponential convergence to the minimizer $F^*$ with some rate $c = {\mu\alpha} > 0$, that is,
  \begin{align*}
 F(x(t)) - F^* \leq \e^{-ct} (F(x(0)) - F^*).
  \end{align*}
\end{thm}
\begin{proof}
  The proof is deferred to Appendix~\ref{subsec:proof_thm_8}
\end{proof}

\begin{cor}
 If the proximal PL inequality is true for a certain value of $\alpha$, it will also be true for all $\alpha' \in (0,\alpha]$. This follows from the fact that $\mcD_g(x,\alpha') \geq \mcD_g(x,\alpha)$. 
\end{cor}

\begin{remark}
 In~\cite[Theorem 3]{SHM-MRJ:21}, the forward-backward envelope is shown to be globally exponentially stable under the PL conditions. The forward-backward envelope is an upper bound to $f(z) + g(z)$, under the assumption of $L-$smoothness of $f$, where $z = \prox{\alpha g}(x - \alpha \nabla f(x))$.  In contrast, here we are interested in the cost function. \change{The cost function in many optimization problems has some physical interpretation, and hence this result is more intuitive.} 
\end{remark}
\subsection{Extensions to time-varying optimization}
\change{In recent literature~\cite{AD-VC-AG-GR-FB:23f,AA-AS-ED:20,LJ-YZ:14}, variations of the classic problem~\eqref{eq:opt_problem} with time-varying parameters have also been considered. In typical applications, time-varying parameters arise from the evolution of the environment.}
 In continuous time, the goal is to define dynamics that track the minimizer of the optimization problem,
\begin{align}\label{eq:opt_problem_time_varying}
 \min_{x} f(x,\theta(t)) + g(x,\theta(t)) = F(x,\theta(t)).
\end{align}\change{
This formulation is generalizes the typical online optimization setup where $\theta(t) = t$. Under Lipschitz constraints on $\theta$, one may define bounds for the difference of the current cost and optimal cost.  The difference of the  cost incurred and the optimal cost is termed regret, and is a means to evaluate the performance of algorithms. Due to our exponential bound on the cost we may bound the regret at each time.}

\begin{thm}[Cost tracking under the PL condition]\label{thm:PL_time_varying}
  Consider the time-varying optimization
  problem~\eqref{eq:opt_problem_time_varying}. Assume $f(x,\theta(t))$ and
  $g(x,\theta(t))$ satisfy the proximal PL
  condition~\eqref{eq:Prox_PL_condition} with constant $\mu$ at each time
  $t$, $g$ is CCP, and $F(x, \theta)$ is differentiable with respect to,
  and Lipschitz continuous in $\theta$, with parameter $\ell_\theta$. Given
  an initial condition $x(0)$ and a parameter trajectory $\theta(t)$,
  define ${V(t) = F(x(t),\theta(t)) - F(\xstar(t),\theta(t))}$. Then, under
  the proximal gradient dynamics~\eqref{eq:prox_dynamics}
  \begin{enumerate}
  \item $\displaystyle D^+ V(t) \leq -\mu\alpha V(t) + 2 \ell_\theta
    \norm{\dot \theta(t)}{}$, and
  \item the Gr\"{o}nwall inequality for Dini derivatives implies
    \begin{align*}
      V(t) \leq \e^{-\mu\alpha t} V(0) + \frac{2\ell_\theta}{\mu\alpha}
      \int_0^t \e^{-\mu\alpha(t - \tau)}\norm{\dot \theta(\tau)}{} d\tau.
    \end{align*}    
  \end{enumerate}
\end{thm}
\begin{lcss}
\begin{proof}
  \change{The proof follows from Lipschitz arguments and is available in the arxiv version\footnote{\href{https://arxiv.org/pdf/2409.10664}{https://arxiv.org/pdf/2409.10664}}.}
\end{proof}
\end{lcss}
\begin{arxiv}
\begin{proof}
 Through arguments similar to the proofs of Theorems~\ref{thm:decreasing-energy-smooth} and for any $s_1 \in \partial_x F(x,\theta), s_2 \in \partial_x F(\xstar,\theta)$,
  \begin{align}
 D^+V(x,t) &\leq s_1^\top \dot x + (\nabla_\theta F(x,\theta) - \nabla_\theta F(\xstar,\theta))^\top \dot \theta \nonumber\\
    & \quad - s_2^\top \dot{x}^\star(t).
  \end{align}
 Following the arguments of Theorem~\ref{thm:PL_cts_time}, the first term is bounded above by $-\alpha\mu(V(x,t))$. Since $F$ is Lipschitz in $\theta$, $\norm{\nabla_{\theta} F(x)}{} \leq \ell_\theta$. Through an application of Cauchy Schwarz and triangle inequality, the second term may be bounded by $2\ell_\theta \norm{\dot \theta}{}$
 The final term is $0$ due to the optimality of $\xstar$. In summary, we have 
  \begin{align*}
 D^+ V(x,t) \leq -\mu\alpha V(x,t) + 2 \ell_\theta \norm{\dot \theta}{}.
  \end{align*}
 Using Gr\"{o}nwall's lemma, we obtain the second result.
\end{proof}
\end{arxiv}
\section{Applications and Examples}\label{sec:apps}

\subsubsection*{Example \#1: Least Squares with \(\ell_1\) regularization}

Through the proximal gradient dynamics, one may solve the $\ell_1$ regularized least squares problem, otherwise known as the Least Absolute Shrinkage and Selection Operator (LASSO). The problem is given by
\begin{align}
 \min_{x \in \R^n} \frac{1}{2}\norm{Ax - u}{2}^2 + \lambda\norm{x}{1},\label{eq:lasso}
\end{align}
where $A\in \R^{m\times n},\; u \in \R^m$ 
The corresponding proximal gradient dynamics, referred to as the ``firing rate competitive network"~\cite{VC-AG-AD-GR-FB:23a} are given by
\begin{align}
 \dot x = -x + \softtbig{\alpha\lambda}{(I_n - \alpha A^\top A)x + \alpha A^\top u}.\label{eq:lasso_prox}
\end{align}
Here, $\softt{\alpha\lambda}{\cdot}$ represents the soft thresholding
operator and is the proximal operator of the $\ell_1$ norm:
\begin{align*}
 \softt{\alpha\lambda}{x_i} = \begin{cases}
 x_i - \alpha\lambda\sign{x_i} & \text{if} \quad |x_i| > \alpha\lambda, \\
    0 & \text{if} \quad |x_i| \leq \alpha\lambda
  \end{cases}
\end{align*}
The paper~\cite{VC-AG-AD-GR-FB:23a} also includes a bound on the
convergence to the optimal value of $x$ under the restricted isometry
property.
\begin{cor}
  Under the proximal gradient dynamics~\eqref{eq:lasso_prox} for the LASSO
  problem~\eqref{eq:lasso}, the cost function decays to the optimal value
  exponentially.
\end{cor}
\begin{proof}
  The LASSO problem is known to satisfy the PL
  condition~\cite{HK-JN-MS:16}. The result follows from
  Theorems~\ref{thm:PL_cts_time}.
\end{proof}

\subsubsection*{Example \#2: Matrix based problems}
 Next, consider matrix optimization problems. In this setting, the
analogue of sparse vectors are low rank matrices. However, similar to the
$\ell_0$ norm, imposing a low rank constraint leads to nonconvexity. A convex relaxation for the $\ell_0$ norm is the trace norm. The
proximal operator of $g(X) = \norm{X}{*}$ is
\begin{align}
  \prox{\alpha g}(X) = \Softt{\alpha}{X} =  U\Sigma_\alpha V^\top,
\end{align}
where $X\in \R^{m\times n}, X = U\Sigma V^\top$ is an SVD and
$\Sigma_{\alpha\lambda}$ is diagonal with $(\Sigma_{\alpha\lambda})_{ii} =
\softt{\alpha\lambda}{\Sigma_{ii}}$~\cite{AB:17}. A general class of
problems including \emph{matrix recovery} and \emph{matrix completion
problems} can be written as
\begin{align}\label{eq:matrix-recovery}
  \min_{X\in \R^{m\times n}}  \frac{1}{2}\norm{y - \mathcal{A}[X]}{2}^2 + \lambda\norm{X}{*}
\end{align}
where ${\mathcal{A}[X] = [\tr(A_1^\top X), \tr(A_2^\top X),\cdots \tr(A_k^\top X)]^\top}$, with each $A_i \in \R^{m\times n}$
and ${y \in \R^k}$. The proximal gradient dynamics for this problem are
\begin{align}\label{eqn:mat_recon}
  \dot X = - X + \Softtbig{\alpha\lambda}{X - \alpha \sum\nolimits_{i=1}^k(\tr(A_i^\top X) - y_i)A_i},
\end{align}
A second popular matrix based optimization problem is the \emph{matrix
factorization problem}, with applications in
overparameterized linear neural networks with mean square error
loss~\cite{ZX-HM-ST-EM-RV:23}. In general, this problem is nonconvex. To
promote a low-rank solution, it is customary to add terms proportional to
the nuclear norm and write the problem as,
\begin{align}\label{eq:matrix-factor}
  \min_{W_1,W_2} \frac{1}{2} \norm{Y - XW_1W_2}{F}^2 + \lambda\norm{W_1}{*} + \lambda\norm{W_2}{*},
\end{align}
where $Y\in \R^{N\times m}, X\in \R^{N\times n}, W_1 \in \R^{n \times h}$
and $W_2 \in \R^{h \times m}$. The proximal gradient dynamics are
\begin{subequations}\label{eqn:mat_factor}
  \begin{equation}
    \dot W_1 = -W_1 + \Softtbig{\alpha\lambda}{X - \alpha X^\top(Y - XW_1W_2)W_2^\top},
  \end{equation}
  \begin{equation}
    \dot W_2 = -W_2 + \Softtbig{\alpha\lambda}{X - \alpha W_1^\top X^\top(Y - XW_1W_2)}
  \end{equation}
\end{subequations}
Both~\eqref{eqn:mat_recon} and~\eqref{eqn:mat_factor} are guaranteed to
converge to a minimizer or a saddle point via
Theorem~\ref{thm:decreasing-energy-smooth}. It is an open question whether the costs in~\eqref{eq:matrix-recovery} and~\eqref{eq:matrix-factor} are exponentially decreasing along trajectories of~\eqref{eqn:mat_recon} and~\eqref{eqn:mat_factor}, respectively. 
\change{
\subsubsection*{Example \#3: Sparse neural network training}
We conduct a numerical experiment using proximal gradient dynamics to train a feedforward neural network with two hidden layers (dimensions 8 and 2) and a one-dimensional output with $\tanh(\cdot)$ activation on each layer. The network is trained on the scipy moons dataset ($N=200$) for classification, using an $\ell_1$ regularizer and binary cross-entropy loss. For the purposes of visualization, we vary two first-layer weights, freezing the other weights, and show the evolution of the dynamics in Figure~\ref{fig:loss_landscape}, overlaid on a loss landscape. Since $f$ is differentiable and $g$ is CCP, we observe that the cost is decaying, verifying Theorem~\ref{thm:decreasing-energy-smooth}.\footnote{We provide additional implementation details on  \href{https://github.com/AnandGokhale/Proximal-Gradient-Dynamics}{Github}} These dynamics could easily be implemented for the purposes of training the whole network. Implementations of such continuous time algorithms on analog circuits could reduce power consumption.}

\begin{figure}
  \centering
  \includegraphics[width=0.45\textwidth]{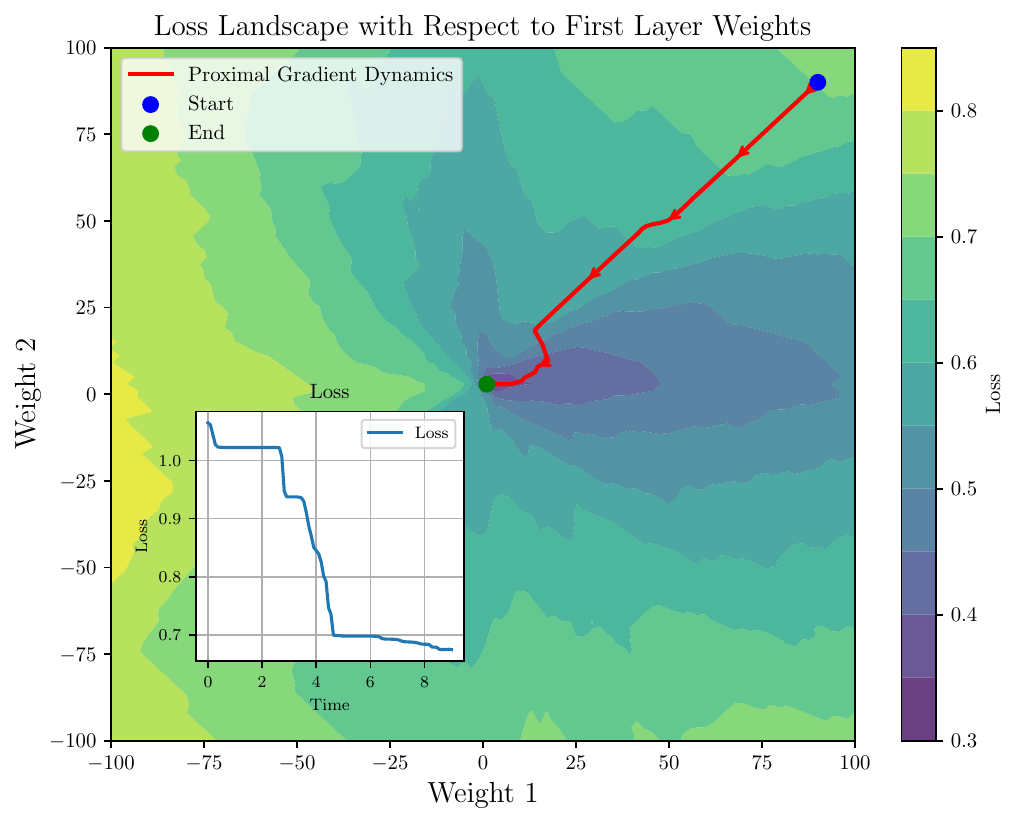} 
  \caption{\change{Nonconvex loss landscape of a regularized feed-forward neural network with the trajectory of proximal gradient dynamics. The trajectory goes along a path where the cost function is monotonically decreasing, since the loss function is differentiable and the nonsmooth regularizer is CCP.}}
  \label{fig:loss_landscape}
\end{figure}
\section{Discussion}\label{sec:conclusion}
In this letter, we discuss cost based convergence properties of the proximal gradient dynamics. We prove in Theorem~\ref{thm:decreasing-energy-smooth} that the cost function of the associated optimization problem is nonincreasing for finite-valued regularizers. For non-finite-valued regularizers, we are able to show that the cost is nonincreasing in Theorem~\ref{thm:decreasing-energy-nonsmooth}. Next, we extend the PL conditions to continuous-time proximal gradient dynamics, explicitly showing that the cost function decays exponentially under the PL condition. We use this result to consider the case of parametric optimization problems with time-varying parameters, and provide a cost function tracking bound. 

An interesting direction of future work is the application of proximal gradient dynamics to design continuous-time solvers, such as analog circuits and biologically plausible neural networks. Further, these dynamics may also allow a generalized nonsmooth analysis of systems based on energy landscapes. A system that is currently studied under this framework is the human brain~\cite{SG-MC-BB-SFM-STG-FP-DSB:18}. While previous work~\cite{SCS-JPP-WG-KB:20} suggests that the algorithms we design for modeling a brain must be gradient-based, due to our analysis proximal gradient dynamics, one could design biologically plausible algorithms for a nonsmooth energy landscape.

\appendix
\subsection{Technical Results}


\begin{lemma}[\change{Proximal optimality condition~\cite{NP-SB:14}}]\label{lem:prox-optimality}
  Let $\map{g}{\real^n}{\realextended}$ be CCP. Then for $\change{v} \in \real^n, \alpha > 0$, define $z := \prox{\alpha g}(v)$. Then
  $
        \alpha^{-1}\bigl(v - z\bigr) \in \partial g(z).$
  \end{lemma}
  

\begin{lemma}[Forward-invariance of $\sashachange{\dom F}$]\label{lemma:fwd-invariance}
  The domain of $F$ is forward-invariant under the proximal gradient dynamics~\eqref{eq:prox_dynamics}. In other words, for all ${t \geq 0}$, ${x \in \dom F}$ implies ${\phi_t(x) \in \dom F}$.
\end{lemma}
\begin{proof}
 \begin{lcss}
The proof is a consequence of Nagumo's theorem\change{~\cite[Theorem~3.1]{FB:99}} since $\prox{\alpha g}(x - \alpha \nabla f(x)) \in \dom g$ for all $x \in \real^n$, and $\dom g = \dom F$.
\end{lcss}
\begin{arxiv}
 First, note that $\dom F$ is a closed and convex set and that $\dom F = \dom g$. Nagumo's theorem~\cite[Theorem~3.1]{FB:99}, states that the $\dom g$ is forward invariant provided that $\Fprox(x)$ lies in the tangent cone to $\dom F$ at $x$ for all $x \in \dom F$. By definition of $\prox{}$, we have that $\prox{\alpha g}(x-\alpha \nabla f(x)) \in \dom g = \dom F$ for all $x \in \real^n$ and that $z - x$ lies in the tangent cone of $\dom F$ at $x$ for $z,x \in \dom F$. Thus, the result is proved. 
\end{arxiv}
\end{proof}

\begin{lemma}\label{lem:ccp_subdifferential_continuity}
	Let $\map{g}{\real^n}{\real}$ be a CCP function and let $\{x_k\}_{k=1}^\infty \subseteq \real^n$ be a sequence such that $x_k \to x$ as $k \to \infty$. Then for any $v \in \real^n$,
	\begin{equation}\label{eq:technical-lemma}
		\sashachange{\limsup_{k \to \infty} s_{x_k}^\top v \leq \max\setdef{s^\top v}{s \in \partial g(x)},}
	\end{equation}
	\sashachange{where the sequence $\{s_{x_k}\}_{k=1}^\infty$ is given by $s_{x_k} = \argmin{s \in \real^n} \|s\|_2^2$ s.t. $s \in \partial g(x_k)$.}
\end{lemma}
\begin{proof}
	\sashachange{Since $g$ is convex and finite-valued, it is locally Lipschitz and thus $\partial g$ is bounded on any compact set~\cite[Prop.~16.20]{HHB-PLC:17}. Since $x_k \to x$, we can assume that $\{x_k\}$ lies in a compact set for sufficiently large $k$. Thus, since $s_{x_k} \in \partial g(x_k)$, we also know that $\{s_{x_k}\}$ is bounded as well for sufficiently large $k$. 
	We study the set of cluster points of $\{s_{x_k}\}$ which correspond exactly to the set of points which are the limit of some subsequence of $\{s_{x_k}\}$. We know that at least one cluster point exists by the Bolzano-Weierstrass Theorem. Thus, let $s$ be a cluster point of $\{s_{x_k}\}$. Then since $s$ is the limit of a subsequence of $\{s_{x_k}\}$, and the subdifferential of a convex function is outer semicontinuous~\cite[Prop.~16.36]{HHB-PLC:17}, we know that $s \in \partial g(x)$. In other words, every cluster point of $\{s_{x_k}\}$ is in $\partial g(x)$. And since the $\limsup$ of a sequence is equal to the greatest cluster point, we conclude~\eqref{eq:technical-lemma}.}
\end{proof}
\subsection{Proof of Theorem~\ref{thm:PL_cts_time}}~\label{subsec:proof_thm_8}

  Note that if $V(x(0))$ is infinity, there is nothing to show. If $V(x(0))$ is finite, \change{then $V(x(t))$ is finite for all $t \geq 0$ by Lemma~\ref{lemma:fwd-invariance}}. Similar to the proof of Theorem~\ref{thm:decreasing-energy-smooth}, consider the Dini derivative of ${V(x(t)) =  f(x(t)) + g(x(t)) - F^*}$. Using the arguments presented in Theorem~\ref{thm:decreasing-energy-smooth} upto Equation~\eqref{eq:useful_for_PL},
  \begin{align}
    D^+V(x) \leq  \nabla f(x)^\top(z - x) + g(z) - g(x).\label{eq:start_exp_bound}
       \end{align}  
     where $z = \prox{\alpha g}(x - \alpha \nabla f(x))$. Also, note that
  \begin{align}
    z 
         &= \argmin{u} \bigl\{  g(u) {-} g(x) + \frac{1}{2\alpha}\norm{u - x}{}^2 + \langle\nabla f(x), u {-} x\rangle \bigr\}. \nonumber
     \end{align}
  Therefore,
     \begin{align}
         &\nabla f(x)^\top(z - x) + g(z) - g(x) + \frac{1}{2\alpha} \norm{x - z}{2}^2  \nonumber\\
         &\quad= -\frac{\alpha}{2} \mcD_g(x, \alpha). \nonumber
     \end{align}
  Adding and subtracting $\frac{1}{2\alpha} \norm{x - z}{2}^2$ in~\eqref{eq:start_exp_bound},
     \begin{align}
  D^+V(x) &\leq -\frac{\alpha}{2}\mcD_g(x, \alpha) - \frac{1}{2\alpha} \norm{x - z}{2}^2 .
     \end{align}   
  Using the proximal PL condition~\eqref{eq:Prox_PL_condition},
     \begin{align}
  D^+V(x) &\leq -\alpha\mu (F(x) - F(\xstar)) - \frac{1}{2\alpha} \norm{x - z}{2}^2 .
     \end{align}  
  Now, the second term may be bounded above by $0$. Using Gr\"{o}nwall's lemma we get our required result.

  \bibliography{alias, Main, FB,New}

\end{document}